\documentclass{amsart}

\usepackage[all]{xy}
\usepackage{hyperref}
\hypersetup{
    colorlinks=true,       % false: boxed links; true: colored links
    linkcolor=blue,          % color of internal links
    citecolor=blue,        % color of links to bibliography
    filecolor=blue,      % color of file links
    urlcolor=blue           % color of external links
}
\usepackage[backrefs,lite]{amsrefs}
\usepackage{tikz}
\usepackage{todonotes}

\newcommand{\vdim}{\operatorname{vdim}}
\newcommand{\ttvdim}{\operatorname{tvdim}}

\newcommand{\edim}{\operatorname{edim}}
\newcommand{\ttedim}{\operatorname{tedim}}

\newcommand{\rank}{\operatorname{rank}}

\newcommand{\Hom}{\operatorname{Hom}}

\newcommand{\Aut}{\operatorname{Aut}}
\newcommand{\Cl}{\operatorname{Cl}}
\newcommand{\Spec}{\operatorname{Spec}}
\newcommand{\Id}{\operatorname{Id}}
\newcommand{\Cox}{\operatorname{Cox}}

\newcommand{\tedim}{\operatorname{tedim}}
\newcommand{\WDiv}{\operatorname{WDiv}}
\newcommand{\PDiv}{\operatorname{PDiv}}
\newcommand{\Irr}{\operatorname{Irr}}

\newcommand{\pp}{\mathbb{P}}

\newcommand{\qq}{\mathbb{Q}}
\newcommand{\zz}{\mathbb{Z}}

\newcommand{\kk}{\mathbb{K}}

\newtheorem{introthm}{Theorem}

\newtheorem{theorem}{Theorem}[section]

\newtheorem{proposition}[theorem]{Proposition}
\newtheorem{corollary}[theorem]{Corollary}

\theoremstyle{definition}

\newtheorem{notation}[theorem]{Notation}
\newtheorem{definition}[theorem]{Definition}
\newtheorem{example}[theorem]{Example}

\newtheorem{remark}[theorem]{Remark}

\theoremstyle{remark}

\numberwithin{equation}{section}

\usepackage[norelsize]{algorithm2e}

\begin{document}

\title{On a Notion of Toric Special Linear Systems}
\author[J.~Moraga]{Joaqu\'in Moraga}
\address{
Department of Mathematics, University of Utah, 155 S 1400 E, 
Salt Lake City, UT 84112}
\email{moraga@math.utah.edu}


\maketitle

\medskip 

\begin{abstract}
In this note, we study linear systems on 
complete toric varieties $X$
with an invariant point whose orbit under the action of $\Aut(X)$
contains the dense torus $T$ of $X$.
We give a characterization of such varieties in terms
of its defining fan and introduce a new definition of expected dimension
of linear systems which consider the contribution
given by certain toric subvarieties.
Finally, we study degenerations of linear systems on these toric varieties 
induced by toric degenerations.
\end{abstract}

\setcounter{tocdepth}{1}
\tableofcontents

\section*{Introduction}

Given a complete algebraic variety $X$, 
the problem of determining the dimension
of the linear system of hypersurfaces of $X$
of a given divisor class passing through finitely
many points in very general position
with prescribed multiplicities is an open problem in algebraic geometry.
In the case of $X=\pp^2$, the Segre-Harbourne-Gimigliano-Hirschowitz conjecture ~\cite{Gi,Hi,Ha,Se}
predicts the dimension of such linear systems.
For $X=\pp^3$ there is an analogous conjecture (see  ~\cite{LU1,LU2}).
The case $X=\pp^n$ is studied in ~\cite{BDP}, 
where the authors prove that the linear subspaces 
passing through the general points can give contribution
to the dimension of the linear system when they are contained
in the base locus with certain multiplicities.
There are also results in another kind of varieties,
for example Hirzebruch surfaces and
products of projective lines (see ~\cite{Dum10} and ~\cite{LM14}).
In any variety of dimension equal or greater than two, 
the problem remains open. 

In this note, we study the above problem for $X$ a complete toric variety
admitting an automorphism which maps a very general point into an invariant one.
We will denote by $\Cl(X)$ the divisor class group of $X$.
Given a class $[D]\in \Cl(X)$, $k$ points $p_1, \dots, p_k\in X$
and $k$ non-negative integers $\mu_1, \dots, \mu_k$, we denote
by $\mathcal{L}_{[D]}(p_1^{\mu_1}, \dots, p_k^{\mu_k})$
the linear system of hypersurfaces of divisor class $[D]$
passing through $p_i$ with multiplicity at least $\mu_i$
for each $i\in \{1, \dots, k\}$.
The {\em virtual dimension} of $\mathcal{L}=\mathcal{L}_{[D]}(p_1^{\mu_1}, \dots ,p_k^{\mu_k})$ is
\[
\vdim(\mathcal{L}) =
\dim\left( [D] \right) -
\sum_{i=1}^k \binom{n+\mu_i-1}{n} -1.
\]
The {\em expected dimension} of $\mathcal{L}$
is $\edim(\mathcal{L})=\max(\vdim(\mathcal{L}),-1)$.
We say that $\mathcal{L}$ is {\em special} 
if the inequality $\dim(\mathcal{L})>\edim(\mathcal{L})$ holds,
otherwise we say that the system is {\em non-special}.

In the following, we describe the main results of this paper.
First, we study the existence of automorphisms of toric varieties
mapping points of the dense torus into invariant points.
In this direction, we prove the following theorem.

\begin{introthm}\label{introthm1}
The complete toric variety $X(\Sigma)$ admits an automorphism
which maps an invariant point into the dense torus if and only if
there exists a smooth full-dimensional cone $\sigma \in \Sigma$ such that
\[
\Sigma(1)\setminus \sigma(1) \subset -\sigma.
\]
In this case, $\sigma$ defines such invariant point.
\end{introthm}

We will call such toric varieties {\em quasi-transitive}.
Our next main result concerns linear systems on quasi-transitive toric varieties.
Given a linear system $\mathcal{L}$ on a quasi-transitive toric variety $X(\Sigma)$,
we introduce the {\em toric expected dimension} (see Definition~\ref{deftor}),
denoted by $\ttedim(\mathcal{L})$,
which is the expected dimension of $\mathcal{L}$ considering 
the contribution given by the toric subvarieties passing through the invariant point
whose orbit under $\Aut(X(\Sigma))$ contains $T$.
The second main result of this paper is the following
(see Theorem~\ref{coxlinear} for the complete statement).

\begin{introthm}
Let $p_1,\dots, p_k$ be points in very general position
on a quasi-transitive toric variety $X(\Sigma)$.
Then we have that
\[
\dim(\mathcal{L}) \geq \ttedim(\mathcal{L})\geq \edim(\mathcal{L})
\] 
for $\mathcal{L}=\mathcal{L}_{[D]}(p_1^{\mu_1}, \dots, p_k^{\mu_k})$.
\end{introthm}

We call the linear system {\em toric special} if the first inequality is strict,
and {\em toric non-special} otherwise.
Given a linear system $\mathcal{L}$ on a quasi-transitive toric variety $X(\Sigma)$
and certain toric degeneration of $X(\Sigma)$,
in Theorem~\ref{degeneration}
we relate the toric non-speciality 
of the linear system $\mathcal{L}$ 
with the toric non-speciality 
of the degenerate linear systems.
Finally, we give examples of linear systems
that are special and toric non-special
and linear systems that are toric special.

The paper is organized as follows: In Section~\ref{bs} we recall some definitions and
notation from toric geometry and Cox rings. 
In Section~\ref{sec2} we give a characterization of complete
toric varieties admitting an automorphism which maps an invariant  point
into the torus. Finally, in Section~\ref{sec3} we introduce the concept
of toric non-special linear systems and study toric degenerations
of such linear systems.

\subsection*{Acknowledgements}
The author would like to thank Antonio Laface for many useful conversations.

\section{Basic Setup}\label{bs}

In this section, we shortly recall the standard notation 
of toric varieties. 
The notation of toric geometry follows ~\cite{CLS}
while the notation of Cox rings is the one from ~\cite{ADHL}.

We work over an algebraically closed field $\kk$
of characteristic zero.
Let $N$ be a finitely generated free abelian group and let 
$M=\Hom(N,\zz)$ be the dual of $N$.
We denote by $N_\qq =N\otimes_\zz \qq$ and
$M_\qq = M\otimes_\zz \qq$ the associated rational vector spaces.
Given a pointed convex polyhedral cone $\sigma \subset N_\qq$,
we denote by $\sigma^\vee \subset M_\qq$ its dual cone,
by $\kk[\sigma^\vee\cap M]$ the monoid ring of $(\sigma^\vee, +)$
and by $X(\sigma)=\Spec(\kk[\sigma^\vee \cap M])$
the associated affine toric variety.

We define a {\em fan} $\Sigma$ to be a finite set of pointed
convex polyhedral cones of $N_\qq$ such that the face of any cone
in $\Sigma$ is again in $\Sigma$, and the intersection of two cones in $\Sigma$
is a face of both.
Given a fan $\Sigma$ we can define a toric variety, denoted by $X(\Sigma)$, by gluing 
the affine toric varieties $X(\sigma)$ and $X(\sigma')$ along
$X(\sigma\cap \sigma')$ whenever $\sigma$ and $\sigma'$ are cones of $\Sigma$.
Any $\mu$-dimensional face of $\sigma$ defines a $(n-\mu)$-dimensional orbit 
of $X(\sigma)$. The above gluing is along $T$-invariant 
subvarieties, then $X(\Sigma)$ is also endowed with a $T$-action.
We denote by $\Sigma(\mu)$ the set of $\mu$-dimensional faces of $\Sigma$.

Given an algebraic variety $X$, we denote by
$\WDiv(X)$ the group of Weil divisors of $X$
and 
$\PDiv(X)$ the group of principal divisors of $X$.
For a toric variety $X(\Sigma)$, we denote by
$\WDiv_T(X(\Sigma))$ and $\PDiv_T(X(\Sigma))$ the group of
$T$-invariant Weil divisors and $T$-invariant principal divisors,
respectively.
There is a bijection between the 
irreducible $T$-invariant divisors of $X(\Sigma)$
and the one-dimensional faces of $\Sigma$. 
We denote by $\rho_1,\dots, \rho_r$
the primitive lattice generators of the one-dimensional faces
of $\Sigma$, and by $D_i$ the irreducible $T$-invariant divisor associated
to $\rho_i$ for each $i$.

Given an algebraic variety $X$ which is irreducible,
normal, with only constant invertible functions 
and finitely generated divisor class group,
we define its Cox rings to be
\[
\mathcal{R}(X)=\bigoplus_{[D]\in \Cl(X)}H^0\left( X, \mathcal{O}_X(D)\right).
\]
In Section~\ref{sec2} we give an explicit construction
of the Cox ring of certain toric varieties,
we refer the reader to ~\cite[Chapter 2]{ADHL}
for the general construction.

Finally, we recall the construction of projective toric
varieties from polytopes.
Given a full-dimensional convex compact polytope $P\subset M_\qq$,
we denote by $\Sigma_P$ its dual fan
and $X(P)$ the toric variety associated to the dual fan.
Observe that we can write
\[
P=\{ m\in M_\qq \mid 
\langle m, \rho_i \rangle \leq  -d_i, 
\text{ for $i\in \{1,\dots, r\}$}\}
\]
for certain integers $d_1, \dots, d_r$.
The divisor $D_P=\sum_{i=1}^r d_i D_i$
defines an ample divisor on $X(P)$,
and we say that $D_P$ is the divisor of $X(P)$
associated to the polytope $P$.
Observe that different polytopes can define exactly
the same toric variety if they have the same dual fan, 
but the associated divisors will define different embeddings
into projective spaces.
We say that a polytope is smooth if it defines a smooth toric variety.

\section{Automorphism Group of Toric Varieties}\label{sec2}

In this section, we study toric varieties $X(\Sigma)$
containing an invariant point,
such that there exists an automorphism 
mapping such point into the torus. 
We describe the Cox ring of such toric varieties
in order to study their automorphisms groups.

\begin{definition}
We say that an invariant point of $X(\Sigma)$
is {\em transitive on the torus} if its orbit with respect
to $\Aut(X(\Sigma))$ contains a point of the torus, or equivalently,
its orbit with respect to $\Aut(X(\Sigma))$ contains the torus. 
We say that a cone $\sigma \in \Sigma$ is
{\em transitive on the torus} if $\sigma$ is full-dimensional
and the corresponding invariant point is transitive on the torus.
Given a complete toric variety $X(\Sigma)$
which contains an invariant point which is transitive on the torus
we say that it is a {\em quasi-transitive toric variety}.
\end{definition}

In what follows we will fix a basis 
$\{e_1, \dots, e_n\}$ of $N$,
given an element $\rho \in N_\qq$
we denote by $\rho_i$ its $i$-th coordinate
on this basis.
If $X(\Sigma)$ is a quasi-transitive toric variety
we can assume without loss of generality
that $\sigma=\langle -e_1,\dots, -e_n\rangle$
is in $\Sigma$ and is transitive on the torus.
First, we will describe these toric varieties 
as a quotient of an open subvariety of an affine space
following ~\cite[Section 2.1.1]{ADHL}.

Let $\rho_1, \dots, \rho_r$ denote the primitive generators
of the one-dimensional faces of $\Sigma$, 
with $r\geq n+1$ and $\rho_i=-e_i$ for $1\leq i \leq n$.
Let $F=\zz^r$ and consider the linear map
$P\colon F \rightarrow N$ sending the $i$-th
canonical base vector $f_i\in F$ to $\rho_i \in N$.
Denote by $\delta \subset F_\qq$ the positive orthant
and define a fan $\widehat{\Sigma}$ in $F$ consisting
on all the faces of $\delta$ whose image on $N$
are contained in some cone of $\Sigma$.
Then $P$ induces a toric morphism 

and we have exact sequences 
\begin{equation}\label{exact}
 \xymatrix{
0\ar[r] & M\ar[r]^-{P^*} & E\ar[r]^-{Q} & K\ar[r] & 0 \\
0\ar[r] & K^*\ar[r]^-{Q^*} & F\ar[r]^-{P} & N\ar[r] & 0,\\ 
 }
\end{equation}
where $E$ is the dual of $F$, 
$P^*\colon M\rightarrow E$ the dual map of $P$
and $Q\colon E\rightarrow K=E/P^*(M)$ the induced projection.
Since $\rho_i = -e_i$ for $1\leq i \leq n$, we can write
\[
P= \begin{bmatrix} -\Id_n & P_0 \end{bmatrix}
\qquad
Q=\begin{bmatrix} P_0^t & \Id_{r-n} \end{bmatrix},
\]
where $P_0$ is the $n\times (r-n)$ matrix
whose columns are the vectos $\rho_{n+1}, \dots, \rho_r$.
Observe that $Q$ induces a $K$-grading on $\kk[E\cap \delta^\vee]$,
then by ~\cite[Theorem 2.1.32]{ADHL} we conclude that
\[ \Cox(X(\Sigma)) \simeq  \kk[E\cap \delta^\vee]\] 
as $K$-graded polynomial rings.

Recall that the group $\WDiv^T(X(\Sigma))$
of $T$-invariant Weil divisors of $X(\Sigma)$
is generated by the $T$-invariant divisors $D_i$
corresponding to the rays $\rho_i$ of $\Sigma$.
Then, we have an isomorphism $E\simeq \WDiv^T(X(\Sigma))$
given by 
\[ e \mapsto \langle e,f_1\rangle D_1 +\dots + \langle e , f_r \rangle D_r.\]
Moreover, the injective morphism $P^*$
identifies $M$ with $\PDiv^T(X(\Sigma))$.
Thus, we conclude that $K\simeq \Cl(X(\Sigma))$
and we can write
\begin{equation}\label{cox}
\Cox(X(\Sigma))\simeq k[x_1, \dots, x_r], 
\end{equation}
\begin{equation*}
\deg(x_i) =[D_i] \in \Cl(X(\Sigma)),
\quad
\Spec(\Cox(X(\Sigma)))\simeq \kk^r.
\end{equation*}
Denoting by $H=\Spec(\kk[K])$ the torus acting on 
$\Cox(X(\Sigma))$ we have that the subvariety $X(\widehat{\Sigma})\subset \Spec(\Cox(X(\Sigma)))$
is $H$-invariant and the morphism $p$ is a good quotient for the induced
action on $X(\widehat{\Sigma})$.
In the coordinates~\ref{cox} the complement $V(\Sigma)$
of $X(\widehat{\Sigma})$ in $\kk^r$ 
has defining ideal 
\[
\Irr(\Sigma)=\left\langle \prod_{i\in I} x_i \mid
I\subset \{ 1, \dots, r\} \text{ and $\{ \rho_i \mid i\in I\}$
are not the rays of a cone of $\Sigma$} \right\rangle,
\]
called the {\em irrelevant ideal} of $\Cox(X(\Sigma))$.

Now we turn to describe the automorphisms of $X(\Sigma)$,
which are induced by $H$-equivariant automorphisms
of $X(\widehat{\Sigma})$.
First, observe that any element $t\in T$ defines
an automorphism of $X$
so we can identify $T\subset \Aut(X(\Sigma))$.
We denote by $\Aut(N,\Sigma)$ the subgroup
of automorphism of $N$ preserving the fan $\Sigma$, any such automorphism
induces an automorphism of $X(\Sigma)$. Finally, we say that
$m\in M$ is a {\em Demazure root} of $\Sigma$
if the following condition holds:
\[
\text{ There exists } i \in \{1, \dots, r\} \text{ such that }
\langle m, \rho_i \rangle = -1
\text{ and } 
\langle m, \rho_j \rangle \geq 0, 
\text{ for all $j\neq i$}.
\]
We also say that such $m$ is a {\em Demazure root of $\rho_i$}.
Observe that given a Demazure root $m$ of $\rho_i$
and $t\in \kk^*$
we have a $K$-graded automorphism of $\Cox(X(\Sigma))$ 
defined by 
\begin{equation}\label{Demroot}
y_{(m,t)}(x_i)=x_i + t \prod_{j\neq i} x_j^{\langle m,\rho_j\rangle}
\text{ and }
y_{(m,t)}(x_j)=x_j \text{ for all $j\neq i$}.
\end{equation}
This automorphism induces an automorphism of $X(\Sigma)$.
We denote by $\mathcal{R}(\Sigma)$ the set of
automorphisms induced by Demazure roots on $X(\Sigma)$.
By abuse of notation we also denote by $\mathcal{R}(\Sigma)$
the set of $K$-graded automorphisms of $\Cox(X(\Sigma))$ defined by~\ref{Demroot}.
With the above notation, we state the following theorem
proved in ~\cite[Corollary 4.7]{Cox92}.

\begin{theorem}\label{coxaut}
Let $X(\Sigma)$ be a complete simplicial toric variety. 
Then $\Aut(X(\Sigma))$ is generated
by $T,\mathcal{R}(\Sigma)$ and $\Aut(N,\Sigma)$.
\end{theorem}

Given a cone $\sigma \in \Sigma$
we denote by $\mathcal{O}_\sigma$ the corresponding
$\mathbb{T}$-invariant orbit of $X(\Sigma)$.
We denote by $\Gamma(\sigma)$
the monoid generated by the classes of $\mathbb{T}$-invariant
divisors of $X(\Sigma)$ that do not contain $\mathcal{O}_\sigma$.
We write 
\[
\Upsilon(\Sigma)=\{ \Gamma(\sigma) \mid \sigma \in \Sigma\}.
\]

From ~\cite[Theorem 3.7]{Baz13} we know that 
$\mathcal{O}_\sigma$ and $\mathcal{O}_{\sigma'}$ on $X(\Sigma)$
are contained in the same $\Aut(X(\Sigma))$-orbit if and only if there exists an
automorphism $\phi \colon \Cl(X(\Sigma))\rightarrow \Cl(X(\Sigma))$
with the following properties:
\begin{itemize}
\item $\phi(\Gamma(\mathcal{O}_\sigma))=\Gamma(\mathcal{O}_{\sigma'})$.
\item $\phi(\Upsilon(\Sigma))=\Upsilon(\Sigma)$.
\item There exists a permutation $f$ of elements in $\{1,\dots, r\}$ such that
$\phi([D_i])=[D_{f(i)}]$.
\end{itemize}

\begin{proof}[Proof of Theorem~\ref{introthm1}]
Given a quasi-transitive toric variety $X(\Sigma)$,
without loss of
generality we can assume that $\sigma=\langle -e_1, \dots, -e_n\rangle\in\Sigma$
is transitive in the torus.
With this assumption, it is enough to prove
that all the entries of $P_0$ are non-negative
with respect to the basis $\{e_1,\dots, e_n\}$.

First, we claim that if $P_0$
has a negative entry then 
$\mathcal{O}_\sigma$ and $\mathcal{O}_{0}$
do not lie in the same $\Aut(X(\Sigma))$-orbit
using ~\cite[Theorem 3.7]{Baz13}.
Recall that under this hypothesis 
$[D_{n+1}],\dots, [D_r]$
is a basis of $\Cl(X(\Sigma))$.
Indeed, $\rank\Cl(X(\Sigma))=r-n$
and $[D_{n+1}],\dots, [D_r]$
generate $\Cl(X(\Sigma))$ since
every $[D_i]$ with $i\in \{1,\dots, n\}$
is in the linear span of $[D_{n+1}],\dots, [D_r]$.
Observe that $\Gamma(0)$ is the monoid
generated by $[D_1],\dots, [D_r]$ in $\Cl(X(\Sigma))$
and $\Gamma(\sigma)$ is the monoid
generated by $[D_{n+1}],\dots, [D_r]$.
Assume that some entry of $P_0$ with respect
to the basis $\{e_1,\dots, e_n\}$ is negative,
then $\Gamma(0)$ strictly contains $\Gamma(\sigma)$.
Indeed, $\Gamma(\sigma)$ is generated
by $[D_{n+1}],\dots, [D_r]$ in $\Cl(X(\Sigma))$
as a monoid,
$\Cl(X(\Sigma))$ is freely generated
by $[D_{n+1}],\dots, [D_r]$ as a group.
Moreover, since $X(\Sigma)$ is complete,
we have an exact sequence 
\[
0\rightarrow M \rightarrow \oplus_{i=1}^n \mathbb{Z}D_i \rightarrow \Cl(X(\Sigma))
\rightarrow 0,
\]
where the first map is $e\mapsto {\rm div}(\chi^e)$.
Taking $e$ to be an element of the canonical basis
$\{e_1,\dots, e_n\}$ 
and using the fact that some entry of $P_0$ is negative
with respect to this basis,
 we can see that 
for some $i \in \{1, \dots, n\}$, the class of $[D_i]$ 
is linearly equivalent to a sum
$a_{n+1}[D_{n+1}]+\dots +a_r[D_r]$
where some $a_i$'s is negative.
Therefore, $\Gamma(0)$ contains an element of
the form
$a_{n+1}[D_{n+1}]+\dots +a_r[D_r]$
where some $a_i$'s is negative.
Thus, if there exists an automorphism
$\phi \colon\Cl(X(\Sigma))\rightarrow \Cl(X(\Sigma))$
with $\phi(\Gamma(0))=\Gamma(\sigma)$
then we have that $\phi(\Upsilon(\Sigma))\neq \Upsilon(\Sigma)$,
since $\phi$ maps the maximal monoid $\Gamma(0)$
to the monoid $\Gamma(\sigma)$,
which is strictly contained in $\Gamma(0)$,
leading to a contradiction.
Thus, if $P_0$ has a negative entry,
there is no automorphism $\phi \colon \Cl(X(\Sigma))
\rightarrow \Cl(X(\sigma))$ with 
$\phi(\Gamma(0))=\Gamma(\sigma)$
and $\phi(\Upsilon(\Sigma))=\Upsilon(\Sigma)$,
proving the claim.

Now, assume that all the entries of the matrix $P_0$
are non-negative.
Then, for every point $p\in (\kk^*)^r$ there exists
$t_1, \dots, t_n \in \kk^*$ such that the first $n$ coordinates of
$y_{(e_1,t_1)}\circ \dots \circ y_{(e_n,t_n)}(p)$ vanish, 
this automorphism of the Cox ring induces an automorphism
of $X(\Sigma)$ which maps a point of the torus 
to the invariant point corresponding to $\sigma$.
\end{proof}

\begin{definition}
Given a full-dimensional polytope $P\subset M_\qq$
and a smooth vertex $p\in P$
we define the {\em convex capsule of $p$ in $P$} as follows:
Pick all the one-dimensional faces of $P$ which contain $p$,
since $p$ is a smooth vertex and $P$ is full-dimensional 
we have exactly $n$ one-dimensional faces containing it.
Call $p_1, \dots, p_n$ the vertices of such one-dimensional faces which
are not $p$ and let $H$ be the hyperplane passing through $p_1, \dots, p_n$.
The convex capsule of $p$ in $P$ is then the convex hull of $p,p_1,\dots, p_n$
and the reflection of $p$ with respect to $H$.
Given a smooth vertex $p$ of a polytope $P$, we say that $p$
is {\em transitive on $P$} if its convex capsule in $P$ contains $P$.
\end{definition}

\begin{corollary}\label{pol}
A vertex $p$ in a full-dimensional smooth polytope $P\subset M_\qq$
defines an invariant point which is transitive on the torus 
if and only if $p$ is transitive on $P$.
\end{corollary}

\begin{example}
The first polytope correspond to the blow-up of $\pp^2$
at the three invariant points, the shaded square corresponds
to the convex capsule of one of its vertices, from the picture we conclude that
such blow-up of $\pp^2$ is not quasi-transitive. The second polytope
correspond to the Hirzebruch surface $\Sigma_1$ and from Corollary~\ref{pol}
we deduce that the Hirzebruch surface is quasi-transitive.
Observe that the vertices on the top of the polytope
corresponding to $\Sigma_1$ are transitive on the polytope,
and the vertices on the bottom are not.

\begin{figure}[h]
  \begin{minipage}[h]{.5\linewidth}
    \centering
    \begin{tikzpicture}
\draw[gray,very thin, fill=gray!20!white] (0,0)--(-0.75,0)--(-0.75,-0.75)--(0,-0.75)--(0,0);
\draw[step=0.75cm, gray, very thin ] (-1.5,-1.5) grid (1.5,1.5);
\draw [black,very thick] (-0.75,-0.75) -- (0,-0.75)-- (0.75,0)--(0.75,0.75)--(0,0.75)--(-0.75,0)--(-0.75,-0.75);
    \end{tikzpicture}
  \end{minipage}%
  \begin{minipage}[h]{.5\linewidth}
    \centering
    \begin{tikzpicture}\draw[gray!20!white,very thin, fill=gray!20!white] (-0.75,0)--(-0.75,-0.75)--(-1.5,-0.75)--(-0.75,0);
\draw[step=0.75cm, gray, very thin ] (-2.25,-1.5) grid (1.5,0.75);
\draw [black, very thick, fill=gray!20!white] (-0.75,-0.75) -- (0,-0.75)-- (0.75,0)--(-0.75,0)--(-0.75,-0.75);
\draw[gray,very thin] (0,0)--(0,-0.75);
    \end{tikzpicture}
  \end{minipage}
  \begin{minipage}[h]{.5\linewidth}
    \caption{Blow-up of $\pp^2$}
    \label{fig:square}
  \end{minipage}%
  \begin{minipage}[h]{.5\linewidth}
    \caption{Hirzebruch surface}
    \label{fig:rect}
  \end{minipage}
\end{figure}
\end{example}

\begin{corollary}
Let $X(\Sigma)$ be a complete toric variety,
let $p_1$ and $p_2$ be two points which are transitive on the torus.
If no invariant divisor contains both points, then $X(\Sigma)\simeq \pp^1\times \dots \times \pp^1, 
p_1=[1:0]\times \dots \times [1:0]$ and $p_2=[0:1]\times \dots \times [0:1]$.
\end{corollary}

\begin{proof}
Let $\sigma_1$ be the full-dimensional cone of $p_1$
and $\sigma_2$ the full-dimensional cone of $p_2$.
Since no invariant divisor contains both $p_1$ and $p_2$, 
then the rays of $\sigma_1$ and $\sigma_2$ are disjoint.
Moreover, by Theorem~\ref{introthm1}
we have that 
$- \sigma_1 (1) \subset \sigma_2(1)$ 
and $-\sigma_2(1) \subset \sigma_1(1)$
and both cones are smooth,
then we conclude that $\sigma_1=-\sigma_2$,
applying an automorphism on $N_\qq$ that
maps $\sigma_1$ to $\langle -e_1,\dots, -e_n\rangle$
we conclude the statement.  
\end{proof}

Now, we give further description of the Demazure roots  
of a quasi-transitive toric variety $X(\Sigma)$.
As always, we assume that $\langle -e_1,\dots, -e_n \rangle \in \Sigma$
is transitive on the torus.
By Theorem~\ref{introthm1} we can assume that the matrix $Q$
equals $\begin{bmatrix} P^t_0 & \Id_{r-n}\end{bmatrix}$,
where $P_0$ is a matrix whose entries are non-negative.
Denote by 
\[
I_j =\{ i\mid \text{ the $i$-th column of $Q$ is $e_j$ }\},
\]
for $1\leq j \leq r-n$ and  
\[
I=\{ i \mid 1\leq i \leq n \text{ and $i$ is not contained in some $I_j$}\}.
\]
Observe that any $i\in \{1, \dots, r\}$ is contained in some of such sets.
Moreover, we have that $i\in \{1,\dots, n\}$ belongs to $I_j$
exactly when the $(j+n)$-th column of $P$ has
Demazure root $-e_i$
and $i\in I$ if and only if no column of $P$ has Demazure root $-e_i$.

\begin{proposition}
Let $X(\Sigma)$ be a quasi-transitive toric variety
with 
\[
\sigma =\langle -e_1,\dots, -e_n \rangle \in \Sigma
\]
being transitive on the torus.
If $\rho\in \Sigma(1) \setminus \sigma(1)$
has a Demazure root $m$, then $m=-e_i$ for some
$i\in \{1,\dots, n\}$. 
\end{proposition}

\begin{proof}
Recall from Theorem~\ref{introthm1}
that the entrie of $\rho$ with respect
to the basis $\{e_1,\dots, e_n\}$
are non-negative.
Since $\langle m, e_i \rangle \leq 0$
for each $1\leq i\leq n$
we conclude that there exists
$j\in \{1, \dots, n\}$ such that
the $j$-th coordinate of $\rho$ is $1$
and $\langle m, e_j \rangle =-1$.
Morever, for each $i\neq j$
we have that the $i$-th coordinate of $\rho$
is zero or $\langle m , e_i \rangle=0$.
Let $i\neq j$ and assume that $\langle m, e_i \rangle <0$,
then we have that the $i$-th coordinate of any ray 
$\rho'\in \Sigma(1)$
with $\rho' \not\in \{1, \dots, n\}$ and $\rho'\neq \rho$,
vanishes, contradicting the completeness of $\Sigma$.
\end{proof}

\begin{remark}
Clearly, the group of $H$-equivariant automorphisms
of $X(\widehat \Sigma)$ acts linearly on the variables
$\{ x_j \mid i \in I_j\}$. 
Moreover, the rays $-e_1, \dots, -e_n$
of a quasi-transitive toric variety $X(\Sigma)$
can have several Demazure roots.
For a ray $-e_i$ with $i\in I$,
if there existst $j$ with $k=|I_j \cap \{1, \dots, n\}|\geq 1$
such that the $(i,j+n)$ entry of the matrix $P$
is greater or equal than $1$, then $-e_i$
has at least $k$ Demazure roots given by
$\{ e_i -e_k \mid k \in I_j \cap \{1, \dots, n\}\}$.
\end{remark}

\section{Linear Systems via Cox Rings and Degenerations}\label{sec3}

In this section, we study linear system of hypersurfaces 
on quasi-transitive toric varieties passing through finitely
many points in very general position with prescribed multiplicities.
Also, we give a new definition
of {\em special linear systems} on quasi-transitive toric varieties.

First, we start recalling how to construct a basis of the Riemann-Roch
space of a $T$-invariant divisor on a quasi-transitive toric variety $X(\Sigma)$.
Recall that we assume that the cone $\langle -e_1, \dots, -e_n \rangle \in \Sigma$
is transitive on the torus.
In this case the divisor classes of the 
$T$-invariant divisors $D_i$ with $i \geq n+1$ generates
$\Cl(X(\Sigma))$. 
Given a class $[D] \in \Cl(X(\Sigma))$ 
we can pick a representative of the class of the form $D=\sum_{i=n+1}^r d_i D_i$,
which we call {\em standard form}.
Given $m\in M$ we denote by 
\[ 
x^m = x^{P^*(m)} \in \kk[x_1^{\pm 1}, \dots, x_r^{\pm 1}]
\] 
where $P^*$ is defined by the exact sequence ~\ref{exact}.
We define an element $x^D \in k[x_1, \dots, x_r]$
by
\[
x^D =\prod_{i=n+1}^r x_i^{d_i},
\]
then we have that 
\[
H^0\left(X(\Sigma), \mathcal{O}_{X(\Sigma)}(D)\right) \simeq
\bigoplus_{m\in P(D)\cap M} \kk x^D x^m,
\]
where 
\[
P(D)= \{ m \in M_\qq \mid \langle m , \rho_i \rangle \leq -d_i, 
\text{ for $i\in \{1, \dots, r\}$}\}.
\]
Observe that all the elements $x^Dx^m$ with
$m\in P(D)\cap M$ are homogeneous of the same degree 
in the Cox ring of $X(\Sigma)$
and in this case the polytope $P(D)$ is contained
in the first orthant of $M_\qq$.
Then, the dimension of the Riemann-Roch space of $D$
is equal to the number of lattice points of $P(D)$.
	
\begin{remark}(Points in very general position)
Given a quasi-transitive toric variety $X(\Sigma)$
and $k$ points $p_1, \dots, p_k \in X(\Sigma)$,
consider the scheme $X(\Sigma)_{[k]}$
parametrizing $k$-tuples of points in $X(\Sigma)$
and let $\mathcal{P} \in X(\Sigma)_{[k]}$
be the point $p_1+\dots +p_k$.
Given non-negative integers $\mu_1, \dots, \mu_k$
denote by $\mathcal{U}([D], \mu_1, \dots, \mu_k) 
\subset X(\Sigma)_{[k]}$ the open subset where 
$\dim\left( \mathcal{L}_{[D]} (p_1^{\mu_1}, \dots, p_k^{\mu_k})\right)$ 
attains its minimal value.
We denote 
\[
\mathcal{U}= \bigcap_{([D],\mu_1, \dots, \mu_k)}
\mathcal{U}([D], \mu_1, \dots, \mu_k),
\]
where the intersection runs over all the $k$-tuples of 
non-negative integers and all the divisor classes $[D]\in \Cl(X(\Sigma))$.
We say that $p_1, \dots, p_k$ are in {\em very general position}
if the corresponding $\mathcal{P}$ is in $\mathcal{U}$.
Observe that when $p_1, \dots, p_k\in X(\Sigma)$
are points in very general position,
then $\mathcal{L}_{[D]}(p_1^{\mu_1}, \dots, p_k^{\mu_k})$
is independent of $p_1, \dots, p_k$, so we denote such linear
system by $\mathcal{L}_{[D]}(\mu_1, \dots, \mu_k)$.
\end{remark}

\begin{notation}\label{notmat}
Given a quasi-transitive toric variety $X(\Sigma)$,
a divisor in standard form $D$ of $X(\Sigma)$, 
$\mu_1, \dots, \mu_k$ non-negative integers
and $k$ points $p_1,\dots, p_k \in (\kk^*)^n$
in very general position,
we will construct a matrix denoted by
$M(D,\mu_1, \dots, \mu_k)$.

For $\mu \in \zz_{\geq 0}$ we denote
\[
\Delta(\mu )=\{ u \in M \mid 
\langle u , e_1+\dots +e_n \rangle \leq \mu
\text{ and } \langle u,e_i\rangle \geq 0 \text{ for each $i$}\}.
\]
For each $u\in \Delta(\mu)$ we denote
\[
\frac{d}{d x^u} =
\frac{d^{\langle u,e_1 \rangle +\dots + \langle u, e_n\rangle}}{ d x_1^{\langle u, e_1 \rangle}\dots d x_n^{\langle u, e_n \rangle}}.
\]
Given a lattice point $m\in P(D)\cap M$
we denote by $y^m$ the monomial $x^Dx^m$ evaluated
at $x_{n+1}=\dots =x_r=1$.
We define $M(D,\mu_i)$
to be the matrix whose columns are indexed
by $P(D)\cap M$, its rows are indexed
by $\Delta(\mu_i)$, and the entry corresponding
to $(m,u)$ is 
\[
\frac{d y^m }{d x^u}\Big|_{p_i}.
\]
Then we define
\[
M(D,\mu_1, \dots, \mu_k) = 
\begin{bmatrix}
M(D,\mu_1) \\
M(D,\mu_2) \\
\vdots \\
M(D,\mu_k)
\end{bmatrix}.
\]
Given a divisor $D$ of $X(\Sigma)$, we denote by
$M(D, \mu_1,\dots, \mu_k)$ the matrix\\ $M(D', \mu_1, \dots, \mu_k)$,
where $D'$ is in standard form and linearly equivalent to $D$.
\end{notation}

\begin{definition}\label{deftor}
Given a quasi-transitive toric variety $X(\Sigma)$, a divisor
$D$ of $X(\Sigma)$ in standard form and $\mu \in \zz_{\geq 0}$
we denote by $\Delta(D,\mu)=\Delta(\mu)\setminus P(D)^c$.
The {\em toric virtual dimension} of $\mathcal{L}$ is 
\[
\ttvdim(\mathcal{L})=\dim([D])
-\sum_{i=1}^k \Delta(D,\mu_i)-1.
\]
The {\em toric expected dimension} of $\mathcal{L}$
is $\ttedim(\mathcal{L})=\max(\ttvdim(\mathcal{L}),-1)$.
We say that $\mathcal{L}$ is {\em toric special}
if the inequality $\dim(\mathcal{L})>\ttedim(\mathcal{L})$ holds,
otherwise we say that the system is {\em toric non-special}.
\end{definition}

\begin{theorem}\label{coxlinear}
Let $X(\Sigma)$ be a quasi-transitive toric variety $X(\Sigma)$,
$D$ be a divisor of $X(\Sigma)$ in standard form
and $\mu_1, \dots, \mu_k$ non-negative integers.
Then 
\[\dim \left( \mathcal{L}_{[D]}(\mu_1, \dots, \mu_k)\right)
= \rank \left( M(D,\mu_1, \dots, \mu_k) \right).\]
In particular, the inequalities
\[
\dim(\mathcal{L})\geq \tedim(\mathcal{L})\geq \edim(\mathcal{L})
\]
always holds for $\mathcal{L}=\mathcal{L}_{[D]}(\mu_1,\dots, \mu_k)$.
\end{theorem}

\begin{proof}
Recall that an hypersurface of $X(\Sigma)$
whose divisor class is $[D]$ and pass through
$k$ very general points $p_1, \dots, p_k \in X(\Sigma)$
with multiplicities $\mu_1, \dots,\mu_k$ is defined 
by a polynomial in $\kk[x_1, \dots, x_r]$
of degree $[D]$ which is contained in the ideal 
\begin{equation}\label{idealmul}
\cap_{i=1}^k \left( I(p^{-1}(p_i))^{\mu_k} \colon J(\Sigma)^\infty\right),
\end{equation}
where $I(p^{-1}(p_i))$ denote the ideal on $\Cox(X(\Sigma))$
defining the fiber over $p_i$.
Moreover, we have that
\[
I(p^{-1}(p_i))=\left\langle 
x_j -p_{i,j}\prod_{k=n+1}^r x_k^{\rho_{k,j}} \mid
j\in\{1, \dots, n\}
\right\rangle,
\]
for $i\in \{1, \dots, k\}$,
where $p_{i,j}$ and $\rho_{k,j}$ denotes the $j$-th 
coordinate of $p_i$ and $\rho_{k}$ respectively.
Observe that the fibers are irreducible.
We claim that \[(I(p^{-1}(p_i))^{\mu_k} \colon J(\Sigma)^\infty)=I(p^{-1}(p_i))^{\mu_k}\]
for each $i$ and $\mu_k$.
First we prove the case $\mu_k=1$. If $J(\Sigma)^\mu f \subset I(p^{-1}(p_i))$
for some element $f\in \kk[x_1, \dots, x_r]$ and $\mu \in \zz_{\geq 0}$, 
then the vanishing set of the ideal generated by $J(\Sigma)^\mu f$
is contained in $V(x_{n+1} \dots x_r)\cup V(x_1)\cup \dots \cup V(x_n) \cup V(f)$, 
and since $p^{-1}(p_i)$ is irreducible and is not contained in some $V(x_i)$
or in $V(x_{n+1}\dots x_r)$,
we conclude that $p^{-1}(p_i) \subset V(f)$, proving the claim.
Now, we proceed by induction on $\mu_k$ for some fixed $i$.
Clearly we have the inclusion $\supset$,
for the other inclusion pick \[ f\in (I(p^{-1}(p_i))^{\mu_k}\colon J(\Sigma)^\infty)
\subset (I(p^{-1}(p_i))^{\mu_{k-1}}\colon J(\Sigma)^\infty) = I(p^{-1}(p_i))^{\mu_{k-1}},\]
then we can write $f=gh$ for some $g\in I(p^{-1}(p_i))^{\mu_{k-1}}$ 
and $h\in \kk[x_1, \dots, x_r]$.
Thus, for some $\mu \in \zz_{\geq 0}$ we have that 
$J(\Sigma)^\mu gh \subset I(p^{-1}(p_i))^{\mu_k}$. 
Taking the derivatives up to the order $\mu_k-1$ with
respect to the variables $x_1,\dots, x_n$ of
the elements of $J(\Sigma)^\mu gh$, the above implies
that $h\in (I(p^{-1}(p_i)) \colon J(\Sigma)^\infty)= I(p^{-1}(p_i))$ 
concluding the claim.

Observe that a polynomial $f\in \kk[x_1, \dots, x_n]$
belongs to 
\[
\cap_{i=1}^k \left( I(p^{-1}(p_i))\right)^{\mu_k}
\]
if and only if its derivatives up to order $\mu_i-1$
with respect to $x_1,\dots, x_n$ vanishes at
\begin{equation}\label{eva}
x_j=p_{i,j} \prod_{k=n+1}^r x_k^{\rho_{k,j}},
\end{equation}
for $j\in \{1, \dots, n\}$ and $i\in \{1, \dots, k\}$.
Now, we construct a matrix whose rank is the 
dimension of $\mathcal{L}_{[D]}(\mu_1, \dots, \mu_k)$.
Let $M'(D,\mu_i)$ be the matrix whose columns
are indexed by $P(D)\cap M$, 
its rows are indexed by $\Delta(\mu_i)$
and the entry corresponding to $(m,u)$ is the element
\[
\frac{d y^m}{d x^u},
\]
evaluated at ~\ref{eva}.
Observe that the entries of $M'(D,\mu_i)$ are elements of
the ring $\kk[x_{n+1},\dots, x_r]$.
Then we define
\[
M'(D,\mu_1, \dots, \mu_k) = 
\begin{bmatrix}
M'(D,\mu_1) \\
M'(D,\mu_2) \\
\vdots \\
M'(D,\mu_k)
\end{bmatrix}.
\]
Now, it suffices to show that $M=M(D,\mu_1, \dots, \mu_k)$
and $M'=M'(D,\mu_1, \dots, \mu_k)$ have the same rank.
Since both matrices have the same size, we can identify its squares submatrices.
Observe that if a square submatrix $N'$ of $M'$ 
has determinant equal to $0$, then the corresponding square submatrix $N$
of $M$ has determinant equal to $0$, since the matrix 
$M$ is obtained from $M'$ by evaluating $x_{n+1}=\dots =x_r=1$.
On the other hand, if $N$ is a square submatrix of $M$ whose determinant
is zero, then the determinant of the corresponding square matrix $N'$ of $M'$
is a polynomial in $\kk[x_{n+1}, \dots, x_r]$ which
vanishes for $x_{n+1}=p_{n+1}, \dots, x_r=p_r$, where $(p_{n+1},\dots,p_r)\in \kk^{r-n}$
is very general, then the determinant of $N'$ is zero.

For the second statement, 
observe that the rows of the matrix $M(D,\mu_1, \dots, \mu_k)$
corresponding to $u\in M\cap P(D)^c$ are null, 
then the first inequality holds 
and the second inequality holds by definition.
\end{proof}

\begin{example}
In ~\cite{LM14} there is a complete study of 
the toric speciality in the case of $(\pp^1)^n$.
In this case, the toric contribution through
a point $p_i$ correspond to the fibers
of the projections $(\pp^1)^n \rightarrow (\pp^1)^k$,
with $k<n$, passing through the point $p_i$.
In this case, there is an explicit formula to compute
$\ttedim(\mathcal{L})$.
\end{example}

From now, we will turn to study linear systems 
in smooth projective varieties. $P$ will denote a smooth polytope
and $X(P)$ the associated toric variety.
We denote by $\mathcal{L}_P(\mu_1, \dots, \mu_k)$ the linear
system of hypersurfaces of $X(P)$ of divisor class corresponding
to the ample divisor determined by $P$
passing through $k$ very general points with multiplicities 
$\mu_1, \dots, \mu_k$ respectively.
Morever, we denote by $M(P, \mu_1,\dots, \mu_k)$
the matrix introduced in ~\ref{notmat} corresponding to the 
divisor associated to $P$ in $X(P)$.

\begin{definition}
Given a polytope $P$, we say that it is in 
{\em standard form} if it is a full-dimensional polytope
contained in the positive orthant of $M_\qq$,
the origin its a transitive smooth vertex of $P$,
and all the one-dimensional vertices of $P$ passing through
the origin are contained in the lines $\qq e_i$
for some $i\in \{1,\dots, n\}$.
\end{definition}

\begin{notation}\label{notdeg}
Let $P$ be a smooth polytope 
in standard form.
Given $i\in \{1, \dots, n\}$ and $c\in \zz_{\geq 0}$
we define the polytopes 
\[
P^+_{c-1} =\{m \in P \mid m_i \geq c-1\}, \qquad
P^+_c = \{m\in P \mid m_i \geq c\},
\]
\[
P^-_{c-1} =\{m \in P \mid m_i \leq c-1\}, \qquad
P^-_c = \{m\in P \mid m_i \leq c\}.
\]
Observe that the origin is transitive on
$P^-_c$ and $P^-_{c-1}$,
while the point $ce_i$ is transitive on $P^+_c$
and $(c-1)e_i$ is transitive on $P^+_{c-1}$.
Given $s\in \{1, \dots, k\}$ we denote
\[
\mathcal{L}^-=\mathcal{L}_{P^-_{c-1}}(\mu_1, \dots, \mu_s),
\qquad
\mathcal{L}^+=\mathcal{L}_{P^+_c}(\mu_{s+1},\dots, \mu_k).
\] 
Finally, given $\mu,c\in \zz_{\geq 0}$ we denote by
\[
\Delta(c,\mu)=\left\{
m\in M \mid \langle m , e_i \rangle \geq c \text{ for each $i$ and } \sum_{j\neq i} \langle m, e_j\rangle  + ( \langle m, e_i \rangle -c)\leq \mu \right\}.
\]
\end{notation}

\begin{remark}
The linear systems $\mathcal{L}^+$ and
$\mathcal{L}^-$ arise naturally from toric degenerations
of the linear system
$\mathcal{L}_P(\mu_1, \dots, \mu_k)$.
Indeed, 
we have a toric degeneration $X(\mathcal{P}_{c-1})\rightarrow \kk$
whose fiber over $t\in \kk^*$ is isomorphic to $X(P)$ and
the fiber over $0  \in \kk$ is isomorphic to the union
$X(P^-_{c-1})\cup X(P^+_{c-1})$ glued along
$X(P^-_{c-1}\cap P^+_{c-1})$. 
Given the linear system $\mathcal{L}_P(\mu_1, \dots, \mu_k)$
we can specialize the first $s\leq k$ of such points
generically to $X(P^-_{c-1})$ and the other $k-s$
points generically to $X(P^+_{c-1})$ to obtain
a degeneration of the linear system into 
\[
\mathcal{L}_{P^-_{c-1}}(\mu_1, \dots, \mu_s)
\text{ and }
\mathcal{L}_{P^+_{c-1}}(\mu_{s+1},\dots, \mu_k).
\]
Analogously, we can degenerate $\mathcal{L}$ to obtain
\[
\mathcal{L}_{P^-_{c}}(\mu_1, \dots, \mu_s)
\text{ and }
\mathcal{L}_{P^+_{c}}(\mu_{s+1},\dots, \mu_k).
\]
\end{remark}

\begin{theorem}\label{degeneration}
Using the notation of ~\ref{notdeg},
assume that the following conditions hold:
\begin{itemize}
\item $\mathcal{L}^+$ and $\mathcal{L}^-$
are toric non-special and $(\ttvdim(\mathcal{L^+})+1)(\ttvdim(\mathcal{L^-})+1)\geq 0$,
\item $\Delta(c,\mu_i) \subset P^+_c$ for any $i\in \{1, \dots, s\}$,
\item $\Delta(0,\mu_i)\subset P^-_{c-1}$ for any $i\in \{s+1, \dots, k\}$.
\end{itemize}
Then $\mathcal{L}_P(\mu_1, \dots, \mu_k)$ is toric non-special.
\end{theorem}

\begin{proof}
Consider the matrix $M(P, \mu_1, \dots, \mu_k)$
and reorder the rows and columns 
such that the first columns corresponds to the points of $P^+_c$
and the first rows corresponds to the conditions of the points
$p_1, \dots, p_s$.
Writting 
\[ p_i = ( p_{i,1}, \dots, p_{i,n} ) \in \kk^n\] 
for the coordinates
of $p_i$, we can consider the entries of the matrix 
$M(P,\mu_1, \dots, \mu_k)$ as polynomials in 
\[
\mathcal{R}= \kk[p_{1,1}, \dots, p_{1,n}, \dots, p_{k,1}, \dots, p_{k,n}].
\]
We define a grading on $\mathcal{R}$ defined by
$\deg(p_{i,j})=1$ if $j=n$ and $i\geq s+1$
and $\deg(p_{i,j})=0$ otherwise.
Moreover, we can write
\[
M(P, \mu_1, \dots, \mu_r) =
\begin{bmatrix} M_1 & K_1 \\
K_2  & M_2
\end{bmatrix}
\]
and applying Theorem~\ref{coxlinear}
to the toric varieties $X(P^+_c)$ and $X(P^-_{c-1})$ with the corresponding
divisor associated to $P^+_c$ and $P^-_{c-1}$ respectively, we conclude that 
$M_1$ and $M_2$
have maximal rank.
Without loss of generality
we assume that we can obtain a square submatrix 
$M'_1$ of $M_1$
and a square submatrix $M'_2$ of
$M_2$ by deleting columns.
We denote by $M'$ the square submatrix of $M(P,\mu_1, \dots, \mu_r)$
obtained by deleting such columns,
then we can write
\[
M' =
\begin{bmatrix} M'_1 & K'_1 \\
K'_2  & M'_2
\end{bmatrix}.
\]
We have that $\deg( \det( M'_2))>\deg(\det (B))$
for any square submatrix $B$ of $[K'_2 \quad M'_2]$.
Thus, by Laplace expansion with respect to the rows
of $M_1$ we conclude that $\deg( \det(M')) = \deg( \det(M'_1)\det(M'_2))>0$,
and the result follows.
\end{proof}

\begin{example}
Let $X(\Sigma_1)$ be the Hirzebruch surface of 
Picture~\ref{fig:rect},
consider the linear system
$\mathcal{L}=\mathcal{L}_{[nD_3 +mD_4]}(m+1,\dots, m+1)$
passing through $\lfloor \frac{n}{m+1} \rfloor$ points in very general position,
where $D_3$ is the $\mathbb{T}$-invariant divisor corresponding to the ray
$(1,1)$ and $D_4$ is the $\mathbb{T}$-invariant divisor corresponding
to the ray $(0,1)$. 
Using Theorem~\ref{degeneration} we can see that $\mathcal{L}$
is special but is toric non-special.
Moreover, we can conclude 
that each point contributes $1$ to the speciality,
and is given by the image, via an automorphism, of a $\mathbb{T}$-invariant curve
passing through the very general point.

In $\mathbb{P}^n$ the toric non-speciality coincide with the standard non-speciality, 
therefore any special linear system on $\mathbb{P}^n$ will be toric non-special 
as well.

In $(\mathbb{P}^1)^n$ the toric contribution corresponds to the fibers
of the morphisms $(\mathbb{P}^1)^n\rightarrow (\mathbb{P}^1)^s, 1\leq s\leq n-1$.
In this case, the linear system $\mathcal{L}_{[D_1+\dots +D_7]}(3,3,3)$ on $(\mathbb{P}^1)^7$
is toric special, the difference between the dimension and the toric expected dimension is $1$
and this contribution is given by certain non-toric subvariety passing through the $3$ points.
\end{example}


\begin{bibdiv}
\begin{biblist}

\bib{ADHL}{book}{
   author={Arzhantsev, Ivan},
   author={Derenthal, Ulrich},
   author={Hausen, J{\"u}rgen},
   author={Laface, Antonio},
   title={Cox rings},
   series={Cambridge Studies in Advanced Mathematics},
   volume={144},
   publisher={Cambridge University Press, Cambridge},
   date={2015},
   pages={viii+530},
   isbn={978-1-107-02462-5},
   review={\MR{3307753}},
}

\bib{Baz13}{article}{
   author={Bazhov, Ivan},
   title={On orbits of the automorphism group on a complete toric variety},
   journal={Beitr. Algebra Geom.},
   volume={54},
   date={2013},
   number={2},
   pages={471--481},
   issn={0138-4821},
   review={\MR{3095734}},
   doi={10.1007/s13366-011-0084-0},
}

\bib{BDP}{article}{
AUTHOR = {Brambilla, Maria Chiara},
AUTHOR = {Dumitrescu, Olivia},
AUTHOR = {Postinghel, Elisa},
TITLE = {On a notion of speciality of linear systems in $\pp^n$},
    journal={Trans. Amer. Math. Soc.},
     YEAR = {2014},
   pages={1--27},
   doi={10.1090/S0002-9947-2014-06212-0},    
}

\bib{Cox92}{article}{
   author={Cox, David A.},
   title={Erratum to ``The homogeneous coordinate ring of a toric variety''
   [MR1299003]},
   journal={J. Algebraic Geom.},
   volume={23},
   date={2014},
   number={2},
   pages={393--398},
   issn={1056-3911},
   review={\MR{3166395}},
   doi={10.1090/S1056-3911-2013-00651-7},
}

\bib{CLS}{book}{
   author={Cox, David A.},
   author={Little, John B.},
   author={Schenck, Henry K.},
   title={Toric varieties},
   series={Graduate Studies in Mathematics},
   volume={124},
   publisher={American Mathematical Society, Providence, RI},
   date={2011},
   pages={xxiv+841},
   isbn={978-0-8218-4819-7},
   review={\MR{2810322 (2012g:14094)}},
   doi={10.1090/gsm/124},
}

\bib{Dum10}{article}{
   author={Dumnicki, Marcin},
   title={Special homogeneous linear systems on Hirzebruch surfaces},
   journal={Geom. Dedicata},
   volume={147},
   date={2010},
   pages={283--311},
   issn={0046-5755},
   review={\MR{2660581 (2012a:14011)}},
   doi={10.1007/s10711-009-9455-1},
}

\bib{Gi}{book}{
   author={Gimigliano, Alessandro},
   title={On linear systems of plane curves},
   note={Thesis (Ph.D.)--Queen's University (Canada)},
   publisher={ProQuest LLC, Ann Arbor, MI},
   date={1987},
   pages={(no paging)},
   isbn={978-0315-38458-3},
   review={\MR{2635606}},
}

\bib{Ha}{article}{
   author={Harbourne, Brian},
   title={Complete linear systems on rational surfaces},
   journal={Trans. Amer. Math. Soc.},
   volume={289},
   date={1985},
   number={1},
   pages={213--226},
   issn={0002-9947},
   review={\MR{779061 (86h:14030)}},
   doi={10.2307/1999697},
}

\bib{Hi}{article}{
   author={Hirschowitz, Andr{\'e}},
   title={Une conjecture pour la cohomologie des diviseurs sur les surfaces
   rationnelles g\'en\'eriques},
   language={French},
   journal={J. Reine Angew. Math.},
   volume={397},
   date={1989},
   pages={208--213},
   issn={0075-4102},
   review={\MR{993223 (90g:14021)}},
   doi={10.1515/crll.1989.397.208},
}


\bib{LU2}{article}{
   author={Laface, Antonio},
   author={Ugaglia, Luca},
   title={On a class of special linear systems of $\Bbb P^3$},
   journal={Trans. Amer. Math. Soc.},
   volume={358},
   date={2006},
   number={12},
   pages={5485--5500 (electronic)},
   issn={0002-9947},
   review={\MR{2238923 (2007e:14009)}},
   doi={10.1090/S0002-9947-06-03891-8},
}


\bib{LU1}{article}{
   author={Laface, Antonio},
   author={Ugaglia, Luca},
   title={Standard classes on the blow-up of $\Bbb P^n$ at points in very
   general position},
   journal={Comm. Algebra},
   volume={40},
   date={2012},
   number={6},
   pages={2115--2129},
   issn={0092-7872},
   review={\MR{2945702}},
   doi={10.1080/00927872.2011.573517},
}

\bib{LM14}{article}{
author={Laface, Antonio},
author={Moraga, Joaqu\'in},
title={Linear systems on the blow-up of $(\pp^1)^n$},
journal={Linear Algebra and its Applications},
volume={492},
pages={52-67},
year={2016},
issn={0024-3795},
doi={http://dx.doi.org/10.1016/j.laa.2015.11.009},
}

\bib{Se}{article}{
   author={Segre, Beniamino},
   title={Alcune questioni su insiemi finiti di punti in geometria
   algebrica. },
   language={Italian},
   conference={
      title={Atti Convegno Internaz. Geometria Algebrica},
      address={Torino},
      date={1961},
   },
   book={
      publisher={Rattero, Turin},
   },
   date={1962},
   pages={15--33},
   review={\MR{0146714 (26 \#4234)}},
}

\end{biblist}
\end{bibdiv}
\end{document}